\documentclass[12pt]{amsart}

\usepackage{enumerate, amsmath, amsthm, amsfonts, amssymb, xy,  mathrsfs, graphicx, paralist}
\usepackage[usenames, dvipsnames]{color}
\usepackage[margin=1in]{geometry} 
\usepackage[bookmarks, colorlinks=true, linkcolor=blue, citecolor=blue, urlcolor=blue]{hyperref}
\usepackage{tabu}

\numberwithin{equation}{section}
\newtheorem{theorem}{Theorem}
\numberwithin{theorem}{section}

\newtheorem{proposition}[theorem]{Proposition}
\newtheorem{lemma}[theorem]{Lemma}
\newtheorem{corollary}[theorem]{Corollary}

\theoremstyle{definition}
\newtheorem{rmk}[theorem]{Remark}
\newenvironment{remark}[1][]{\begin{rmk}[#1] \pushQED{\qed}}{\popQED \end{rmk}}
\newtheorem{eg}[theorem]{Example}

\newtheorem{defn}[theorem]{Definition}

\newcommand{\Area}{\operatorname{Area}}

\newcommand{\bA}{\mathbf{A}}

\newcommand{\bC}{\mathbf{C}}

\newcommand{\cE}{\mathcal{E}}
\newcommand{\fE}{\mathfrak{E}}

\newcommand{\bG}{\mathbf{G}}

\newcommand{\bP}{\mathbf{P}}

\newcommand{\bQ}{\mathbf{Q}}

\newcommand{\bR}{\mathbf{R}}

\newcommand{\cX}{\mathcal{X}}

\newcommand{\cY}{\mathcal{Y}}

\newcommand{\bZ}{\mathbf{Z}}

\renewcommand{\phi}{\varphi}

\newcommand{\ol}[1]{\overline{#1}}

\makeatletter
\def\Ddots{\mathinner{\mkern1mu\raise\p@
\vbox{\kern7\p@\hbox{.}}\mkern2mu
\raise4\p@\hbox{.}\mkern2mu\raise7\p@\hbox{.}\mkern1mu}}
\makeatother

\DeclareMathOperator{\Aut}{Aut}

\DeclareMathOperator{\sgn}{sgn}

\usepackage{stmaryrd}

\newcommand{\tors}{\mathrm{tors}}

\DeclareMathOperator{\val}{val}

\title{Counting elliptic curves with prescribed torsion}
\author{Robert Harron}
\address[Harron]{
Department of Mathematics\\
Van Vleck Hall\\
University of Wisconsin--Madison\\
Madison, WI 53706\\
USA
}
\email{rharron@math.wisc.edu}

\author{Andrew Snowden}
\address[Snowden]{
Department of Mathematics\\
East Hall\\
University of Michigan\\
Ann Arbor, MI 48109\\
USA
}
\email{asnowden@umich.edu}
\date{\today}

\subjclass[2010]{Primary: 11G05; Secondary: 14G05, 14G25}

\hypersetup{pdfauthor={Robert Harron and Andrew Snowden},pdftitle={Counting elliptic curves with prescribed torsion},pdfkeywords={Elliptic curves, Rational points}}

\begin{document}

\begin{abstract}
Mazur's theorem states that there are exactly 15 possibilities for the torsion subgroup of an elliptic curve over the rational numbers. We determine how often each of these groups actually occurs. Precisely, if $G$ is one of these 15 groups, we show that the number of elliptic curves up to height $X$ whose torsion subgroup is isomorphic to $G$ is on the order of $X^{1/d}$, for some number $d=d(G)$ which we compute.
\end{abstract}

\maketitle
\tableofcontents
\vskip 1em

\section{Introduction}

The arithmetic of elliptic curves has long been a central area of study in number theory. One of the first general results, due to Mordell \cite{mordell} (see also \cite[Ch.~VIII]{silverman}), states that the group of rational points of an elliptic curve is finitely generated. In other words, if $E/\bQ$ is an elliptic curve then we can write $E(\bQ)=\bZ^r \times E(\bQ)_{\tors}$, where $r \ge 0$ is an integer (the \emph{rank} of $E$), and $E(\bQ)_{\tors}$ is a finite group (the \emph{torsion subgroup} of $E$). Given this result, one would like to understand what the possibilities are for the rank and torsion subgroup. Little is known regarding the rank. On the other hand, work of several mathematicians (most notably Mazur \cite{mazur}) established that there are only 15 possibilities for the torsion subgroup, namely:
\begin{equation}
\label{mazur}
\begin{aligned}
\bZ/N\bZ & \quad \textrm{with $1 \le N \le 10$ or $N=12$} \\
\bZ/2\bZ \times \bZ/N\bZ & \quad \textrm{with $N=2,4,6,8$.}
\end{aligned}
\end{equation}
With this classification in hand, it is natural to ask a more refined question: how often does each of these groups occur? Our purpose here is to provide an answer.

\subsection{The main result}

To make the question precise, we formulate it as a counting problem. Every elliptic curve $E$ over $\bQ$ admits a unique equation of the form
\begin{displaymath}
y^2=x^3+Ax+B
\end{displaymath}
where $A$ and $B$ are integers and $\gcd(A^3,B^2)$ is not divisible by any twelfth powers. We call such an equation \emph{minimal}, and define the (\emph{na\"{i}ve}) \emph{height} of $E$ to be $\max(|A|^3, |B|^2)$. Clearly, there are only finitely many elliptic curves (up to isomorphism) of height $<X$, for any real number $X$. We can therefore introduce a meaningful counting function: if $G$ is one of the groups in \eqref{mazur}, we let $N_G(X)$ be the number of (isomorphism classes of) elliptic curves $E/\bQ$ of height at most $X$ for which $E(\bQ)_{\tors}$ is isomorphic to $G$. The following theorem is our main result.

\begin{theorem}
\label{thm:main}
For any group $G$ in \eqref{mazur}, the limit
\begin{displaymath}
\frac{1}{d(G)}=\lim_{X \to \infty} \frac{\log{N_G(X)}}{\log{X}}
\end{displaymath}
exists. The value of $d(G)$ is as indicated in Table~\ref{f:main}.
\end{theorem}

\begin{table}[!h]
\caption{The values of $d(G)$.}
{\tabulinesep=1.2mm
\begin{tabu}{|c|c||c|c||c|c|}
\hline
$G$ & $d$ & $G$ & $d$ & $G$ & $d$ \\
\hline
\hline
0 & 6/5 & $\bZ/6\bZ$ & 6 & $\bZ/12\bZ$ & 24 \\
\hline
$\bZ/2\bZ$ & 2 & $\bZ/7\bZ$ & 12 & $\bZ/2\bZ \times \bZ/2\bZ$ & 3 \\
\hline
$\bZ/3\bZ$ & 3 & $\bZ/8\bZ$ & 12 & $\bZ/2\bZ \times \bZ/4\bZ$ & 6 \\
\hline
$\bZ/4\bZ$ & 4 & $\bZ/9\bZ$ & 18 & $\bZ/2\bZ \times \bZ/6\bZ$ & 12 \\
\hline
$\bZ/5\bZ$ & 6 & $\bZ/10\bZ$ & 18 & $\bZ/2\bZ \times \bZ/8\bZ$ & 24 \\
\hline
\end{tabu}}
\label{f:main}
\end{table}

Since $d(0)<d(G)$ for all non-trivial $G$, we recover the following well-known result due to Duke \cite{duke}.

\begin{corollary}
Almost all elliptic curves over $\bQ$ have trivial torsion.
\end{corollary}

\subsection{Refined results}
\label{refined}

We actually prove a stronger result than Theorem~\ref{thm:main}: given $G$ as in \eqref{mazur}, there exist positive constants $K_1$ and $K_2$ such that
\begin{displaymath}
K_1 X^{1/d(G)} \le N_G(X) \le K_2 X^{1/d(G)}
\end{displaymath}
holds for all $X \ge 1$. This suggests that the limit
\begin{displaymath}
c(G)=\lim_{X \to \infty} \frac{N_G(X)}{X^{1/d(G)}}
\end{displaymath}
might exist. We prove this is the case for $\#G\leq3$.

\begin{theorem}
\label{thm:asymptotics}
Write $c_N$ for $c(\bZ/N\bZ)$. Then the limits defining $c_1$, $c_2$, and $c_3$ exist, and
\begin{align*}
c_1 &= \frac{4}{\zeta(10)} \approx 3.9960, \\
c_2 &= \frac{1}{\zeta(6)} \left( 2\log(\alpha_-/\alpha_+)+\tfrac{4}{3} (\alpha_++\alpha_-) \right) \approx 3.1969, \\
c_3 &= \frac{1}{\zeta(4)} \left( 2I_+-2I_-+\frac{1}{3}\log\left(\displaystyle\frac{\beta_0\beta_1\beta_5}{\beta_2\beta_3\beta_4}\right)+\frac{9}{4}\left(\beta_0^4+\beta_2^4+\beta_4^4-\beta_1^4-\beta_3^4-\beta_5^4\right) \right) \approx 1.5221,
\end{align*}
where the $\alpha$'s and $\beta$'s are algebraic numbers and the $I$'s are hyperelliptic integrals (see \S\ref{sec:asymptotics}).
\end{theorem}

The case $N=1$ is straightforward; see \cite[Lemma~4.3]{brumer}. The case $N=2$ was previously carried out in \cite[\S2]{grant}. The case $N=3$ does not seem to occur in the literature. We, in fact, obtain power-saving error terms (see Theorem~\ref{lem:main_asymp_lemma} for this more precise version).

\subsection{Interpretation of $d(G)$}
\label{ss:interp}

If $f \colon \bP^1 \to \bP^1$ is a degree $d$ map, then the number of points in $f(\bP^1(\bQ))$ up to height $X$ is approximately $X^{2/d}$. This suggests that $d(G)$ should be related to the degree of the map $f \colon \cX(G) \to \cX(1)$, where here $\cX(-)$ denotes the appropriate moduli stack. When $\cX(G)$ is a scheme, Table~\ref{f:main} shows that
\begin{displaymath}
d(G)=\tfrac{1}{4} \deg(f)=\tfrac{1}{2} \deg(|f|),
\end{displaymath}
where $|f|$ denotes the map of coarse spaces. More generally, as long as $G \ne 0$, we have
\begin{displaymath}
d(G) = \tfrac{1}{2} \deg(|f|)+\tfrac{1}{2} \nu_2 + \nu_3 + \nu_\infty^\prime,
\end{displaymath}
where $\nu_2$ and $\nu_3$ are the number of elliptic points of orders 2 and 3 on $\cX(G)$ and $\nu_\infty^\prime$ is the number of irregular cusps (in the sense of \cite[\S\S 1.2 \& 2.1]{shimura}). We do not have a conceptual explanation of this formula, so it could simply be a numerical coincidence.

\begin{remark}
If the moduli stacks $\cX(G)$ and $\cX(1)$ were actually schemes then Theorem~\ref{thm:main} would follow immediately from the aforementioned fact about maps of $\bP^1$, and $d(G)$ would be half the degree of $f$. It would be desirable to have a general result for counting points in the image of a map of stacks, from which Theorem~\ref{thm:main} could be deduced; see \S \ref{s:future} for a more precise discussion.
\end{remark}

\subsection{Overview of the proof}

We now go over the main ideas in the proof of Theorem~\ref{thm:main} and make some preliminary reductions.

\subsubsection{The function $N'_G(X)$}

Let $N'_G(X)$ be the number of (isomorphism classes of) elliptic curves $E$ of height at most $X$ such that $E(\bQ)$ contains a subgroup isomorphic to $G$. We prove the following theorem.

\begin{theorem}
\label{thm:main2}
For any group $G$ in \eqref{mazur}, there exist positive constants $K_1$ and $K_2$ such that $K_1 X^{1/d(G)} \le N'_G(X) \le K_2 X^{1/d(G)}$ for all $X \ge 1$.
\end{theorem}

Theorem~\ref{thm:main} follows easily from this. Indeed, we have obvious bounds
\begin{displaymath}
N'_G(X)-\sum_{G \subsetneqq H} N'_H(X) \le N_G(X) \le N'_G(X).
\end{displaymath}
As $d(G)<d(H)$, Theorem~\ref{thm:main2} gives $N'_H(X)/N'_G(X) \to 0$ as $X \to \infty$. Thus $N_G(X)/N'_G(X) \to 1$, and so Theorem~\ref{thm:main} (and the stronger form stated in \S \ref{refined}) follows from Theorem~\ref{thm:main2}.

\subsubsection{Proof of Theorem~\ref{thm:main2} when $2G \ne 0$}
\label{sss:oddg}

Suppose that $G$ is one of the 12 groups in \eqref{mazur} for which $2G \ne 0$. Then there is a \emph{universal elliptic curve} $\cE$ over an open subset of $\bA^1$ equipped with a subgroup isomorphic to $G$. (This is not exactly true for $G=\bZ/3\bZ$, see \S \ref{ss:z3} for details.) The universal property implies that an arbitrary curve $E/\bQ$ admits a copy of $G$ in its rational points if and only if $E$ is isomorphic to $\cE_t$ for some $t \in \bQ$. We can describe $\cE$ by an equation of the form
\begin{displaymath}
y^2=x^3+f(t)x+g(t)
\end{displaymath}
where $f$ and $g$ are polynomials. In \S \ref{s:general}, we prove the following general theorem for counting elliptic curves that appear in such families.

\begin{theorem}
\label{general}
Let $f,g \in \bQ[t]$ be non-zero coprime polynomials of degrees $r$ and $s$, with at least one of $r$ or $s$ positive, and write
\begin{displaymath}
\max \left( \frac{r}{4}, \frac{s}{6} \right)=\frac{n}{m},
\end{displaymath}
with $n$ and $m$ coprime. Assume $n=1$ or $m=1$. Let $\cE$ be the family of elliptic curves defined by
\begin{displaymath}
y^2=x^3+f(t)x+g(t).
\end{displaymath}
Let $N(X)$ be the number of (isomorphism classes of) elliptic curves $E/\bQ$ of height at most $X$ for which $E \cong \cE_t$ for some $t \in \bQ$. Then there exist positive constants $K_1$ and $K_2$ such that
\begin{displaymath}
K_1 X^{(m+1)/12n} \le N(X) \le K_2 X^{(m+1)/12n}
\end{displaymath}
for all $X \ge 1$.
\end{theorem}

In each case of interest, the hypotheses of the theorem are satisfied. Table~\ref{f:params} lists the invariants for the various groups (see \S \ref{explan} for proofs). This establishes Theorem~\ref{thm:main2} in the cases $2G \ne 0$.

\begin{table}
\caption{Data for the universal elliptic curve with $G$-structure.}
{\small\tabulinesep=1.2mm
\begin{tabu}{|c|c|c|c|c|c|}
\hline
$G$ & $r$ & $s$ & $n$ & $m$ & $12n/(m+1)$ \\
\hline
\hline
 3 & 1 & 2 & 1 & 3 & 3 \\
\hline
 4 & 2 & 3 & 1 & 2 & 4 \\
\hline
 5 & 4 & 6 & 1 & 1 & 6 \\
\hline
 6 & 4 & 6 & 1 & 1 & 6 \\
\hline
 7 & 8 & 12 & 2 & 1 & 12 \\
\hline
 8 & 8 & 12 & 2 & 1 & 12 \\
\hline
 9 & 12 & 18 & 3 & 1 & 18 \\
\hline
10 & 12 & 18 & 3 & 1 & 18 \\
\hline
12 & 16 & 24 & 4 & 1 & 24 \\
\hline
$(2,4)$ & 4 & 6 & 1 & 1 & 6 \\
\hline
$(2,6)$ & 8 & 12 & 2 & 1 & 12 \\
\hline
$(2,8)$ & 16 & 24 & 4 & 1 & 24\\
\hline
\end{tabu}}
\label{f:params}
\end{table}

\subsubsection{Proof of Theorem~\ref{thm:main2} when $2G=0$}

If $2G=0$ then there is not a universal curve (over a scheme), and so the above method does not directly apply. When $G$ is trivial or $\bZ/2\bZ$, the situation is simple enough to understand directly (see \S \ref{sec:asymptotics}). When $G=\bZ/2\bZ \times \bZ/2\bZ$, we use a variant of Theorem~\ref{general} to understand $N'_G$ (see \S \ref{sec:Zmod2timesZmod2}). In this way, the remaining cases of Theorem~\ref{thm:main2} are established.

\subsubsection{Proof of Theorem~\ref{general}}
\label{sss:general}

We now comment on the proof of Theorem~\ref{general}, which is the heart of the paper. Recall that if $E_i$ (for $i=1,2$) are elliptic curves over $\bQ$ defined by equations
\begin{displaymath}
y^2=x^3+a_ix+b_i,
\end{displaymath}
then $E_1$ and $E_2$ are isomorphic if and only if there exists a rational number $u \in \bQ$ for which $a_1=u^4 a_2$ and $b_1=u^6 b_2$ \cite[Ch.~III \S 1]{silverman}. We therefore see that an elliptic curve $E/\bQ$ given by an equation
\begin{equation}
\label{eq3}
y^2=x^3+Ax+B
\end{equation}
belongs to a family $\cE$ as above if and only if there exist $u, t \in \bQ$ such that $A=u^4f(t)$ and $B=u^6g(t)$. It follows that $N(X)$ (as in Theorem~\ref{general}) counts the number of pairs $(A,B) \in \bZ^2$ satisfying the following conditions:
\begin{itemize}
\item $4A^3+27B^2 \ne 0$.
\item $\gcd(A^3, B^2)$ is not divisible by any 12th power.
\item $|A|<X^{1/3}$ and $|B|<X^{1/2}$.
\item There exist $u,t \in \bQ$ for which $A=u^4f(t)$ and $B=u^6g(t)$.
\end{itemize}
The first condition ensures that \eqref{eq3} defines an elliptic curve; the second that the equation is minimal; the third that it has height $<X$; and the fourth that it belongs to $\cE$. This reduces the analysis of $N(X)$ to an elementary number theory problem, which we solve directly.

\subsection{Future directions}
\label{s:future}

There are several generalizations of our results that would be of interest. The most immediate, perhaps, is to extend our results to rationally defined subgroups. Precisely, let $G$ be a product of two finite cyclic groups, and let $N^0_G(X)$ be the number of elliptic curves $E/\bQ$ up to height $X$ admitting a subgroup defined over $\bQ$ which (over $\bC$) is isomorphic to $G$. One would then like to compute the limit of $\log{N^0_G(X)}/\log{X}$. Again, there are only finitely many $G$ for which one gets a non-zero answer, namely, those for which the moduli space has genus 0.

It would also be interesting to generalize Theorem~\ref{general} to allow $f$ and $g$ to have a common factor and to remove the restriction on $n$ and $m$. How does this affect the count?

Perhaps the ultimate generalization in this direction is the following. Let $\cX$ and $\cY$ be proper smooth Deligne--Mumford stacks over $\bQ$ with coarse space $\bP^1$, and let $f \colon \cY \to \cX$ be a map. Suppose that there is a good notion of height $h_{\cX}$ on the set $|\cX(\bQ)|$, where $|\cdot|$ denotes isomorphism classes. Then one would like a formula for
\begin{displaymath}
\lim_{T \to \infty} \frac{\# \{ x \in f(|\cY(\bQ)|) \mid h_{\cX}(x) \le T \}}{\log{T}}
\end{displaymath}
in terms of invariants of $\cX$, $\cY$, and $f$ (in the style of \S \ref{ss:interp}).

More generally, one may ask these questions over general global fields. What kind of dependence is there on the base field?

In addition to these generalizations, it would be interesting to complete the results of \S \ref{refined}, and compute the value of $c(G)$ for other $G$'s.

\subsection{Notation}

If $f$ and $g$ are two functions of a real variable $X$, we write $f \lesssim g$ to mean ``there exists a positive constant $c$ such that $f(X) \le c g(X)$ holds for all $X \ge 1$.'' For a prime number $p$, we let $\val_p$ be the usual $p$-adic valuation on $\bQ$ (so $\val_p(p)=1$). For a place $p \le \infty$ of $\bQ$, we let $\vert \cdot \vert_p$ be the usual absolute value on $\bQ$ (so $\vert p \vert_p=p^{-1}$ if $p<\infty$). We often write $\vert \cdot \vert$ in place of $\vert \cdot \vert_{\infty}$. We write $\lfloor x \rfloor$ and $\lceil x \rceil$ for the floor and ceiling of $x \in \bQ$.

\subsection*{Acknowledgements}

We would like to thank Jordan Ellenberg, Wei Ho, and Melanie Matchett Wood for useful conversations, and Bjorn Poonen for pointing out an innaccuracy in an earlier version. We would also like to acknowledge the Sagemath Cloud and Sage \cite{sage} in which we carried out several helpful computations. The first author was supported by NSA Young Investigator Grant \#H98230-13-1-0223 and NSF RTG Grant ``Number
Theory and Algebraic Geometry at the University of Wisconsin''. The second author was supported by NSF Grant DMS-1303082.

\section{Counting elliptic curves in families}
\label{s:general}

The purpose of this section is to prove Theorem~\ref{general}. As explained in \S \ref{sss:general}, it is equivalent to the following proposition.

\begin{proposition}
\label{bd}
Let $f,g \in \bQ[t]$ be coprime polynomials of degrees $r$ and $s$. Assume at least one of $r$ or $s$ is positive. Write
\begin{displaymath}
\max \left( \frac{r}{4}, \frac{s}{6} \right)=\frac{n}{m}
\end{displaymath}
with $n$ and $m$ coprime. Assume $n=1$ or $m=1$. Let $S(X)$ be the set of pairs $(A,B) \in \bZ^2$ satisfying the following conditions:
\begin{itemize}
\item $4A^3+27B^2 \ne 0$.
\item $\gcd(A^3,B^2)$ is not divisible by any 12th power.
\item $|A|<X^{1/3}$ and $|B|<X^{1/2}$.
\item There exist $u,t \in \bQ$ such that $A=u^4 f(t)$ and $B=u^6 g(t)$.
\end{itemize}
Then $X^{(m+1)/12n} \lesssim \# S(X) \lesssim X^{(m+1)/12n}$.
\end{proposition}

\subsection{The upper bound}

We begin by establishing the upper bound. Let $S_1(X)$ be the set of pairs $(u,t) \in \bQ^2$ such that $(A,B)=(u^4f(t), u^6g(t))$ belongs to $S(X)$.

\begin{lemma}
\label{bd1}
For each place $p \le \infty$ of $\bQ$ there exists a constant $c_p>0$ such that for all $t \in \bQ$
\begin{displaymath}
\max(|f(t)|_p, |g(t)|_p) \ge c_p.
\end{displaymath}
If $p$ is a sufficiently large finite prime, one can take $c_p=1$.
\end{lemma}

\begin{proof}
Let $\ol{\bQ}$ be a fixed algebraic closure of $\bQ$, and extend $\vert \cdot \vert_p$ to $\ol{\bQ}$ in any way. Let $\{\alpha_i\}$ be the roots of $f$ and $\{\beta_j\}$ the roots of $g$ in $\ol{\bQ}$. Note that $\alpha_i \ne \beta_j$ for all $i$ and $j$, since $f$ and $g$ are coprime. Let $\delta=\min(|\alpha_i-\beta_j|_p)$. Let $\epsilon>0$ be such that $|f(t)|_p<\epsilon$ implies $|t-\alpha_i|<\delta/2$ for some $i$, and similarly for $g$. If $|f(t)|_p<\epsilon$ and $|g(t)|_p<\epsilon$ then $|t-\alpha_i|_p<\delta/2$ for some $i$ and $|t-\beta_j|_p<\delta/2$ for some $j$, which implies $|\alpha_i-\beta_j|_p<\delta$, a contradiction. We must therefore have $|f(t)|_p \ge \epsilon$ or $|g(t)|_p \ge \epsilon$ for all $t$, and so can take $c_p=\epsilon$.

Now let $p$ be a prime large enough so that: (1) the coefficients of $f$ and $g$ are $p$-integral; (2) the leading coefficients of $f$ and $g$ are $p$-units; and (3) $\alpha_i-\beta_j$ are $p$-units, for all $i$ and $j$. If $|f(t)|_p<1$ then $|t-\alpha_i|<1$ for some $i$. Similarly for $g$. Thus if $|f(t)|_p<1$ and $|g(t)|_p<1$ then $|t-\alpha_i|<1$ and $|t-\beta_j|<1$, which implies $|\alpha_i-\beta_j|<1$, a contradicition. We conclude that we can take $c_p=1$ for such $p$.
\end{proof}

\begin{lemma}
\label{bd2}
For each prime number $p$ there exists a constant $C_p$ with the following property. Suppose $(u,t) \in S_1(X)$. Then
\begin{displaymath}
\val_p(u)=\epsilon+\begin{cases} \lceil -\tfrac{n}{m} \val_p(t) \rceil & \textrm{if $\val_p(t)<0$} \\
0 & \textrm{if $\val_p(t) \ge 0$,} \end{cases}
\end{displaymath}
where $|\epsilon| \le C_p$. Furthermore, one can take $C_p=0$ for $p \gg 0$.
\end{lemma}

\begin{proof}
Suppose $(u,t) \in S_1(X)$, and let $p$ be a prime. Since $A$ and $B$ are integral, we have $4\val_p(u)+\val_p(f(t)) \ge 0$ and $6\val_p(u)+\val_p(g(t)) \ge 0$. Furthermore, $\val_p(u)$ must be minimal subject to these inequalities, or else $p^{12}$ would divide $\gcd(A^3, B^2)$. We can thus write
\begin{displaymath}
\val_p(u)=\max(\lceil -\tfrac{1}{4} \val_p(f(t)) \rceil, \lceil -\tfrac{1}{6} \val_p(g(t)) \rceil).
\end{displaymath}
Equivalently,
\begin{equation}
\label{eq1}
-\val_p(u)=\min(\lfloor \tfrac{1}{4} \val_p(f(t)) \rfloor, \lfloor \tfrac{1}{6} \val_p(g(t)) \rfloor).
\end{equation}

Suppose $\val_p(t)<0$. Let $K_1$ be a constant such that $|\val_p(f(t))-r\val_p(t)|<K_1$ and $|\val_p(g(t))-s\val_p(t)|<K_1$ for all such $t$. (This exists since the functions $\val_p(f(t)/t^r)$ and $\val_p(g(t)/t^s)$ are bounded on $\bQ_p \setminus \bZ_p$.) Then for $\val_p(t)<0$ we have
\begin{displaymath}
\val_p(u)=\epsilon+\max(\lceil -\tfrac{r}{4} \val_p(t) \rceil, \lceil -\tfrac{s}{6} \val_p(t) \rceil)
\end{displaymath}
where $|\epsilon|<K_2$, for some $K_2$. (One can take $K_2=1+\tfrac{n}{m} K_1$.) Since $\val_p(t)<0$, we have
\begin{displaymath}
\max(\lceil -\tfrac{r}{4} \val_p(t) \rceil, \lceil -\tfrac{s}{6} \val_p(t) \rceil)
=\lceil -\tfrac{n}{m} \val_p(t) \rceil,
\end{displaymath}
and so
\begin{displaymath}
\val_p(u)=\epsilon+\lceil -\tfrac{n}{m} \val_p(t) \rceil,
\end{displaymath}
with $|\epsilon|<K_2$.

Now, consider $\val_p(t) \ge 0$. Let $K_3$ be a constant such that $\min(\val_p(f(t)), \val_p(g(t))) \le K_3$ for all such $t$, which exists by Lemma~\ref{bd1}. Appealing to \eqref{eq1}, we find $-\val_p(u) \le K_4$, for an appropirate $K_4$. Let $K_5$ be so that $\val_p(f(t)) \ge K_5$ and $\val_p(g(t)) \ge K_5$ for all $t$ with $\val_p(t) \ge 0$. Then appealing to \eqref{eq1} again, we find $-\val_p(u) \ge K_6$, for an appropriate $K_6$. We thus see that $|\val_p(u)|<K_7$ whenever $\val_p(t) \ge 0$, where $K_7=\max(K_4,K_6)$.

Combining the above two paragraphs, we see that the formula in the statement of the lemma holds with $C_p=\max(K_2, K_7)$.

Now suppose $p$ is large enough so that: (1) the coefficients of $f$ and $g$ are $p$-integral; (2) the leading coefficients of $f$ and $g$ are $p$-unit; and (3) the constant $c_p$ from the previous lemma can be taken to be 1. For $\val_p(t)<0$ we have equalities $\val_p(f(t))=r\val_p(t)$ and $\val_p(g(t))=s\val_p(t)$, which shows that
\begin{displaymath}
\val_p(u)=\max(\lceil -\tfrac{r}{4} \val_p(t) \rceil, \lceil -\tfrac{s}{6} \val_p(t) \rceil)
=\lceil -\tfrac{n}{m} \val_p(t) \rceil.
\end{displaymath}
For $\val_p(t) \ge 0$, we have $\val_p(f(t)) \ge 0$ and $\val_p(g(t)) \ge 0$, with at least one inequality being an equality. Thus $\val_p(u)=0$.  This proves that we can take $C_p=0$ in this case.
\end{proof}

\begin{lemma}
\label{bd3}
There exists a finite set $Q$ of non-zero rational numbers with the following property. Suppose $(u,t) \in S_1(X)$. Then we can write $t=a/b^m$, where $a$ and $b$ are integers such that $b>0$ and $\gcd(a,b^m)$ is not divisible by any $m$th power, and $u=qb^n$, with $q \in Q$.
\end{lemma}

\begin{proof}
Suppose $(u,t) \in S_1(X)$. We then have a unique expression $t=a/b^m$, where $a$ and $b>0$ are integers with $\gcd(a,b^m)$ not divisible by any $m$th power. If $p\mid b$ then $-\val_p(t)=m\val_p(b)-k$, where $0 \le k<m$. If $m=1$ then $k=0$, while if $n=1$ then $0 \le \tfrac{n}{m} k<1$; in either case, $\lceil-\tfrac{n}{m} k \rceil=0$. Thus $\val_p(u)=\epsilon+n\val_p(b)$ by Lemma~\ref{bd2}, with $|\epsilon| \le C_p$. If $p \nmid b$ then $\val_p(t) \ge 0$, and so $|\val_p(u)| \le C_p$. We thus see that $|\val_p(u/b^n)| \le C_p$ for all $p$. Since we can take $C_p=0$ for $p \gg 0$, this means that there are only finitely many possibilities for the rational number $u/b^n$.
\end{proof}

\begin{lemma}
\label{bd4}
Suppose $(u,t) \in S_1(X)$ and write $t=a/b^m$ and $u=qb^n$ per Lemma~\ref{bd3}. Then $|a| \lesssim X^{m/12n}$ and $|b| \lesssim X^{1/12n}$.
\end{lemma}

\begin{proof}
Suppose $(u,t) \in S_1(X)$ and write $t=a/b^m$ and $u=qb^n$ as above. The inequality $\max(|A|^3,|B|^2)<X$ translates to
\begin{displaymath}
|u| \max(|f(t)|^{1/4}, |g(t)|^{1/6})<X^{1/12}.
\end{displaymath}
Let $K_1>0$ be a constant such that $\max(|f(t)|^{1/4}, |g(t)|^{1/6}) \ge K_1$ for all $t$, which exists by Lemma~\ref{bd1}. Then $|u| \le K_1^{-1} X^{1/12}$, and so
\begin{displaymath}
|b| \le K_2 X^{1/12n},
\end{displaymath}
with $K_2=K_1^{-1/n} \max_{q \in Q}(q^{-1/n})$. Suppose for the moment that $|t| \ge 1$. Let $K_3>0$ be a constant so that $K_3^4 |t|^r \le |f(t)|$ and $K_3^6 |t|^s \le |g(t)|$ holds for all such $t$. We thus have
\begin{displaymath}
X^{1/12}>|u| \max(|f(t)|^{1/4}, |g(t)|^{1/6}) \ge K_3 |u| \max(|t|^{r/4}, |t|^{s/6}) = K_3 |u| |t|^{n/m}
=K_4(q) |a|^{n/m},
\end{displaymath}
where $K_4(q)=K_3 q$. We therefore find $|a|<K_5 X^{m/12n}$ with $K_5=\max_{q \in Q}K_4(q)^{-m/n}$. Now suppose $|t|<1$. Then $|a|<|b^m|\leq K_2^m X^{m/12n}$. Thus, in all cases, we have
\begin{displaymath}
|a|<K_6 X^{m/12n}
\end{displaymath}
with $K_6=\max(K_5,K_2^m)$.
\end{proof}

The lemma shows that $\# S_1(X) \lesssim X^{(m+1)/12n}$, which implies the same for $\# S(X)$. We have thus established the upper bound in Proposition~\ref{bd}.

\subsection{The lower bound}

We now turn to the lower bound. Observe that by changing $u$ to $Mu$ for appropriate $M$, it suffices to consider the case where $f$ and $g$ have integer coefficients, which we now assume. For $(a,b) \in \bZ^2$ put $u=b^n$ and $t=a/b^m$, and $A=u^4f(t)$ and $B=u^6g(t)$. Note that $A$ and $B$ are integers. Fix a small constant $\kappa>0$ and let $S_2(X)$ denote the set of pairs $(a,b) \in \bZ^2$ satisfying the following conditions:
\begin{itemize}
\item $a$ and $b$ are coprime and $b>0$.
\item $|a|<\kappa X^{m/12n}$ and $|b|<\kappa X^{1/12n}$.
\item $4A^3+27B^2 \ne 0$.
\end{itemize}
We note that for appropriately chosen $\kappa$, if $(a,b) \in S_2(X)$ then $|A|<X^{1/3}$ and $|B|<X^{1/2}$.

\begin{lemma}
\label{bd5}
There exists a non-zero integer $D$ with the following property: if $(a,b) \in S_2(X)$ then $\gcd(A^3, B^2)$ divides $D$.
\end{lemma}

\begin{proof}
It suffices to find constants $e_p$ such that $\val_p(\gcd(A^3,B^2)) \le e_p$ and $e_p=0$ for $p \gg 0$, for then we can take $D=\prod_p p^{e_p}$.

Suppose $(a,b) \in S_2(X)$, and let $p$ be a prime. Let $K_1$ be a constant so that $|3\val_p(f(t))-3r\val_p(t)| \le K_1$ and $|2\val_p(g(t))-2s\val_p(t)| \le K_1$ for all $t \in \bQ$ with $\val_p(t)<0$. Note that for $p \gg 0$ we can take $K_1=0$. Suppose $\val_p(b)=k>0$, so that $\val_p(t)=-mk$ and $\val_p(u)=nk$. Then
\begin{displaymath}
\val_p(A^3)=\val_p(u^{12} f(t)^3)=12nk-3rmk+\epsilon=12m(\tfrac{n}{m}-\tfrac{r}{4})k+\epsilon
\end{displaymath}
and
\begin{displaymath}
\val_p(B^2)=\val_p(u^{12} g(t)^2)=12nk-2smk+\delta = 12m(\tfrac{n}{m}-\tfrac{s}{6})k+\delta
\end{displaymath}
where $|\epsilon| \le K_1$ and $|\delta| \le K_1$. However, $\tfrac{n}{m}=\max(\tfrac{r}{4}, \tfrac{s}{6})$, and so $\val_p(\gcd(A^3, B^2)) \le K_1$.

Let $K_2$ be a constant so that $\max(3\val_p(f(t)), 2\val_p(g(t))) \le K_2$ for all $t \in \bQ$ with $\val_p(t) \ge 0$. This constant exists by Lemma~\ref{bd1}; furthermore, we can take $K_2=0$ for $p$ sufficiently large. Since $p$ does not divide $b$, this gives $\val_p(\gcd(A^3,B^2)) \le K_2$.

We thus see that we can take $e_p=\max(K_1, K_2)$, and this can be taken to be 0 for $p \gg 0$.
\end{proof}

Let $S_3(X)$ denote the set of pairs $(A,B) \in \bZ^2$ coming from $S_2(X)$. We have a map $S_3(X) \to S(X)$ taking $(A,B)$ to $(A/d^4,B/d^6)$, where $d^{12}$ is the largest 12th power dividing $\gcd(A^3,B^2)$.

\begin{lemma}\label{lem:S3toS}
There exists a constant $N$ such that every fiber of the map $S_3(X) \to S(X)$ has cardinality at most $N$.
\end{lemma}

\begin{proof}
Let $(A_0,B_0) \in S(X)$ be given. Suppose $(A,B) \in S_3(X)$ maps to $(A_0, B_0)$, i.e., $A=d^4A_0$ and $B=d^6B_0$ for some $d \in \bZ$. By Lemma~\ref{bd5}, $d^{12} \mid D$. Therefore, if $N$ is the number of 12th powers dividing $D$, the fibers have cardinality at most $N$.
\end{proof}

\begin{lemma}
There exists a constant $M$ such that every fiber of the map $S_2(X) \to S_3(X)$ has cardinality at most $M$.
\end{lemma}

\begin{proof}
Let $(A,B) \in S_3(X)$ be given. An element $(a,b) \in S_2(X)$ of the fiber gives a solution to the equations $Ax^4=f(y)$ and $Bx^6=g(y)$ by $x=u^{-1}=b^{-n}$ and $y=t=a/b^m$. Furthermore, $(a,b)$ is determined from $(x,y)$. These two equations define plane curves of degree $\max(4,r)$ and $\max(6,s)$, which are easily seen to have no common components. By B\'ezout's theorem, they therefore have at most $M=\max(4,r) \max(6,s)$ points in common, and this bounds the cardinality of the fibers.
\end{proof}

We leave to the reader the standard estimate $X^{(m+1)/12n} \lesssim \# S_2(X)$. This gives the lower bound in Proposition~\ref{bd} since we have a map $S_2(X) \to S(X)$ whose fibers have bounded size. We have thus completed the proof of Proposition~\ref{bd}.

\section{Equations for universal curves}
\label{explan}

Let $G$ be one of the 12 groups in \eqref{mazur} for which $2G \ne 0$. Let $E$ be a family of elliptic curves over a base scheme $S/\bQ$. By a \emph{$G$-structure} on $E$ we mean an injection of groups $i \colon G_S \to E$. When $G=\bZ/N\bZ$, a $G$-structure is just a section of $E$ of order $N$. There is a moduli space $Y/\bQ$ of elliptic curves equipped with a $G$-structure, over which there is a universal elliptic curve $\cE$ with a $G$-structure. (Although for $G=\bZ/3\bZ$ we must slightly change the definition of $Y$, see below.) In fact, $Y$ can be identified with an open subvariety of $\bA^1$; choose such an identification. Then $\cE$ is given by an equation of the form
\begin{equation}
\label{univeq}
y^2=x^3+f(t)x+g(t),
\end{equation}
where $f$ and $g$ are rational functions in $t$ with rational coefficients. Making a change of variables, we can assume that $f$ and $g$ are polynomials and $\gcd(f^3,g^2)$ is not divisible by any 12th powers. Note that neither $f$ nor $g$ is zero, as then every geometric fiber of $\cE$ would have CM. We show that $f$ and $g$ are coprime and compute their degree, which allows us to apply Theorem~\ref{general} to prove Theorem~\ref{thm:main2}.

\subsection{The case $\bZ/4\bZ$}

Let $G=\bZ/4\bZ$. We have the following result.

\begin{proposition}
For an appropriate choice of embedding $Y \to \bA^1$, the polynomials $f$ and $g$ are coprime and satisfy $\deg(f)=2$ and $\deg(g)=3$.
\end{proposition}

\begin{proof}
By \cite[Tab.~3]{kubert}, the curve $\cE$ is given by
\begin{displaymath}
y^2+xy-ty=x^3-tx^2,
\end{displaymath}
for some embedding $Y \to \bA^1$. The result is obtained from the change of variables:
\[
	x\mapsto x+\frac{t}{3}-\frac{1}{12}\quad\text{and}\quad y\mapsto y-\frac{x}{2}+\frac{t}{3}+\frac{1}{24}.
\]
\end{proof}

This establishes the second row of Table~\ref{f:params}, from which Theorem~\ref{thm:main2} follows for $G$.

\subsection{The cases with $\# G>4$}

We assume now that $G$ is one of the 10 groups in \eqref{mazur} of order greater than 4.

\begin{proposition}
\label{prop1}
The polynomials $f$ and $g$ are coprime, and $\deg(f)=4\ell$ and $\deg(g)=6\ell$ for some integer $\ell$. In fact, if $k$ denotes the number of injective group homomorphisms $G \to (\bQ/\bZ)^2$ then $\ell=k/24$.
\end{proposition}

We first require a lemma.

\begin{lemma}
Let $A$ be a DVR with fraction field $K$ and residue characteristic 0. Let $E/K$ be an elliptic curve admitting a $G$-structure. Then $E$ has semi-stable reduction.
\end{lemma}

We give two proofs.

\begin{proof}[First proof]
Suppose $E$ has additive reduction. Let $\ol{E}/A$ be the N\'eron model of $E$. The $G$-structure on $E$ extends to one on $\ol{E}$ by the N\'eron mapping property. Thus the special fiber $\ol{E}_0$ of $\ol{E}$ contains a subgroup isomorphic to $G$. The map $G \to \pi_0(\ol{E}_0)$ necessarily has a kernel, by the classification of special fibers of N\'eron models. Thus $G$ meets the identity component $\ol{E}_0^{\circ}$ non-trivially. But this is a contradiction, since $\ol{E}_0^{\circ} \cong \bG_a$ is torsion-free.
\end{proof}

\begin{proof}[Second proof]
Let $X$ be the moduli space of generalized elliptic curves with $G$-structure, in the sense of Deligne--Rapoport. Then $X$ is a proper scheme. It follows that $E/K$ extends to a generalized elliptic curve $\ol{E}/A$ with $G$-structure. The identity component of the smooth locus of the special fiber of $\ol{E}$ is either an elliptic curve or a torus, and so $E$ has semi-stable reduction.
\end{proof}

\begin{proof}[Proof of Proposition~\ref{prop1}]
By the lemma, the universal curve $\cE$ has semi-stable reduction at all places of $\bP^1_{\bQ}$. Since the equation \eqref{univeq} for $\cE$ is minimal at all finite places (in the sense of \cite[Ch.~VII \S 1]{silverman}), it follows that $f$ and $g$ are coprime (see \cite[Ch.~VII, Prop.~5.1]{silverman} and \cite[Exc.~7.1]{silverman}). Now, \eqref{univeq} is not necessarily minimal at infinity. The minimal equation is given by
\begin{displaymath}
y^2=x^3+t^{-4\ell} f(t)+t^{-6\ell} g(t)
\end{displaymath}
where $\ell$ is minimal so that $t^{-4\ell} f(t)$ and $t^{-6\ell} g(t)$ have non-negative valuation at $\infty$. (Note: the valuation at $\infty$ is $-\deg$.) Now, $Y \subset \bP^1$ is exactly the locus where $\cE$ has good reduction, by universality. (Alternatively, the Deligne--Rapoport theory shows that $\cE$ extends to a genuinely generalized elliptic curve over $\infty$). Thus $\cE$ has multiplicative reduction at $\infty$, and so both $t^{-4\ell} f(t)$ and $t^{-6\ell} g(t)$ have valuation zero at $\infty$ \cite[Exc.~7.1]{silverman}, which exactly says that $\deg(f)=4\ell$ and $\deg(g)=6\ell$.

Now, the $j$-invariant of the family $\cE$ is (up to a scalar)
\begin{displaymath}
\frac{f^3}{4f^3+27g^2}.
\end{displaymath}
We regard this as a self-map of $\bP^1$. Since $f$ and $g$ are coprime, it has degree $12\ell$. On the other hand, the degree of the map $j \colon Y \to \bP^1$ is the number of points in a typical fiber over a complex point. If $E/\bC$ is an elliptic curve, the fiber of $j$ over $E$ is the set of isomorphism classes of pairs $(E, i)$ where $i \colon G \to E$ is a $G$-structure. Generically, $\Aut(E)=\{\pm 1\}$, and so the pairs $(E, i)$ and $(E, -i)$ are isomorphic and there are no other identifications. Thus the fiber of $j$ over $E$ has $k/2$ points, and so $12\ell=k/2$.
\end{proof}

Now, in the notation of Table~\ref{f:params}, we have $r=4\ell$, $s=6\ell$, $n=\ell$, and $m=1$. As $\ell=k/24$, it suffices to compute $k$. When $G=\bZ/N\bZ$, we have $k=N^2 \prod_{p|N} (1-p^{-2})$; and when $G=\bZ/2\bZ \times \bZ/N\bZ$ we have $k=2N^2 \prod_{p \mid N} (1-p^{-2})$. This establishes the row in Table~\ref{f:params} corresponding to $G$. Theorem~\ref{thm:main2} for $G$ then follows from Theorem~\ref{general}.

\subsection{The case $\bZ/3\bZ$}
\label{ss:z3}

We now treat the case $G=\bZ/3\bZ$. There is a minor complication owing to the fact that $\cY_1(3)$ is not a scheme. Since we will treat this group more thoroughly in \S \ref{sec:asymptotics}, we do not include all the details here.

Call a pair $(E, P)$ consisting of an elliptic curve $E$ over a field $k$ of characteristic 0 and a 3-torsion point $P$ \emph{exceptional} if there is a non-trivial automorphism of $E$ over $\ol{k}$ fixing $P$. Let $\rho=e^{2\pi i/3}$, let $E=\bC/\bZ[\rho]$, and let $P$ be the 3-torsion point on $E$ given by $(1-\rho)/3$. Then $(E, P)$ is exceptional. Furthermore, it is the unique exceptional pair over $\bC$, up to isomorphism.

There is a moduli space $Y$ of unexceptional pairs, which can be identified with an open subvariety of $\bA^1$. (In fact, $Y$ is the complement of the single exceptional point in the stack $\cY_1(3)$.) Let $\cE$ be the universal elliptic curve over $Y$.

\begin{lemma}
\label{lem:z3}
Let $E/\bQ$ be an elliptic curve and let $P$ be a rational point of order 3. Then $E$ admits an equation of the form
\begin{displaymath}
y^2+axy+by = x^3
\end{displaymath}
with $a,b \in \bQ$ such that $P=(0,0)$. The pair $(E,P)$ is exceptional if and only if $a=0$.
\end{lemma}

\begin{proof}
The first statement is well-known; see \cite[Tab.~3]{kubert}, for instance. The second is left to the reader
\end{proof}

\begin{proposition}
For an appropriate embedding $Y \to \bA^1$, the universal family $\cE$ is given by $y^2=x^3+f(t)x+g(t)$ with $f(t)=2t-\tfrac{1}{3}$ and $g(t)=t^2+\tfrac{2}{3}t+\tfrac{2}{27}$.
\end{proposition}

\begin{proof}
The previous lemma can, in fact, be applied to $\cE$, and shows that it admits an equation of the form $y^2+2axy+2by=x^3$ where $a$ and $b$ are functions on $Y$, with $a$ invertible. Changing $y$ to $y-(ax+b)$ yields the equation $y^2=x^3+(ax+b)^2$. Changing $(x,y)$ to $(a^2x,a^3y)$ now gives $y^2=x^3+(x+t)^2$ with $t=b/a^3$. Finally, changing $x$ to $x-\tfrac{1}{3}$ yields the stated equation.
\end{proof}

Let $N^1(X)$ (resp.\ $N^2(X)$) be the number of (isomorphism classes of) elliptic curves $E/\bQ$ of height at most $X$ admitting a 3-torsion point $P$ such that $(E,P)$ is not (resp.\ is) exceptional. Combining the above proposition with Theorem~\ref{general}, we see that $X^{1/3} \lesssim N^1(X) \lesssim X^{1/3}$. By Lemma~\ref{lem:z3}, an exceptional curve admits an equation of the form $y^2=x^3+b^2$ with $b \in \bZ$. Clearly then, $N^2(X) \lesssim X^{1/4}$. Finally, we have the obvious bounds $N^1(X) \le N'_G(X) \le N^1(X)+N^2(X)$. This proves $X^{1/3} \lesssim N'_G(X) \lesssim X^{1/3}$, which establishes Theorem~\ref{thm:main2} for $G$.

\section{The group $\bZ/2\bZ \times \bZ/2\bZ$}\label{sec:Zmod2timesZmod2}

We begin with a variant of Proposition~\ref{bd}.

\begin{proposition}
\label{bd-var}
Let $f,g \in \bQ[t]$ be coprime polynomials of degrees $r$ and $s$. Assume one of $r$ or $s$ is positive. Write
\begin{displaymath}
\max \left( \frac{r}{2}, \frac{s}{3} \right)=\frac{n}{m}
\end{displaymath}
with $n$ and $m$ coprime. Assume $n=1$ or $m=1$. Let $S(X)$ be the set of pairs $(A,B) \in \bZ^2$ satisfying the following conditions:
\begin{itemize}
\item $4A^3+27B^2 \ne 0$.
\item $|A|<X^{1/3}$ and $|B|<X^{1/2}$.
\item $\gcd(A^3,B^2)$ is not divisible by any 12th power.
\item There exist $u,t \in \bQ$ such that $A=u^2 f(t)$ and $B=u^3 g(t)$.
\end{itemize}
Define
\begin{displaymath}
h(X) = \begin{cases} X^{(m+1)/6n} & \textrm{if $m+1>n$} \\
X^{1/6} \log(X) & \textrm{if $m+1=n$} \\
X^{1/6} & \textrm{if $m+1<n$} \end{cases}
\end{displaymath}
Then $h(X) \lesssim \# S(X) \lesssim h(X)$.
\end{proposition}

\begin{proof}
We only sketch the proof, as the details are similar to the proof of Proposition~\ref{bd}. We begin with the upper bound. A version of Lemma~\ref{bd2} holds, but now one only has $C_p=1$ for $p \gg 0$ --- the reason for this is that putting an extra $p$ into $u$ only adds two $p$'s to $A$ and three to $B$, and so does not contribute a $p^{12}$ in $\gcd(A^3, B^2)$. The analog of Lemma~\ref{bd3} then says that we can write $t=a/b^m$, with $\gcd(a, b^m)$ not divisible by any $m$th powers, and $u=qcb^n$, where $q$ belongs to a finite set and $c$ is squarefree. The analog of Lemma~\ref{bd4} yields the inequalities $|ca^{n/m}|\lesssim X^{1/6}$ and $|cb^n| \lesssim X^{1/6}$. To count the number of possibilities, note that for any given $c$ there are at most $X^{m/6n}/c^{m/n}$ possibilities for $a$ and $X^{1/6n}/c^{1/n}$ possibilities for $b$, yielding $X^{(1+m)/6n}/c^{(1+m)/n}$ total possibilities. Now integrate over $c$ (up to $X^{1/6}$) to obtain the upper bound.

We now turn to the lower bound. We only treat the case $m+1>n$, as that is the only one we need and the others are a bit more subtle. It turns out that, in this case, we can find enough points with $c=1$ to establish the bound. More precisely, let $\kappa>0$ be a small constant, and consider the set of pairs of coprime integers $(a,b)$ such that $|a^{n/m}|<\kappa X^{1/6}$ and $|b^n|<\kappa X^{1/6}$ and $\Delta \ne 0$. The number of such pairs is $\gtrsim X^{(m+1)/6n}$. The analog of Lemma~\ref{bd5} shows that $\gcd(A^3, B^2)$ is of the form $\alpha^6 \beta$, where $\beta$ divides a fixed integer $D$, and $\alpha$ is square-free. This is sufficient for the application of Lemma~\ref{lem:S3toS}.
\end{proof}

\begin{proposition}
Let $f(t)=-\frac{1}{3} (t^2-t+1)$ and $g(t)=\tfrac{1}{27} (-2t^3+3t^2+3t-2)$. Suppose $E/\bQ$ is an elliptic curve given by an equation $y^2=x^3+Ax+B$. Then all 2-torsion points of $E$ are rational if and only if there exist $u,t \in \bQ$ such that $A=u^2f(t)$ and $B=u^3g(t)$.
\end{proposition}

\begin{proof}
It is well-known that all of the 2-torsion of $E$ is rational if and only if it admits an equation of the form $y^2=x(x-a)(x-b)$ with $a,b \in \bQ$. Changing $x$ to $x+\tfrac{1}{3}(a+b)$, we obtain the equation
\begin{displaymath}
y^2=x^3-\tfrac{1}{3} (a^2-ab+b^2)+\tfrac{1}{27} (-2a^3+3a^2b+3ab^2-2b^3).
\end{displaymath}
We thus see that all of the 2-torsion of $E$ is rational if and only if there exist $a,b,v \in \bQ$ such that
\begin{displaymath}
A=-\frac{v^4}{3} (a^2-ab+b^2), \qquad B=\frac{v^6}{27} (-2a^3+3a^2b+3ab^2-2b^3).
\end{displaymath}
Changing $(a,b)$ to $(a/v^2, b/v^2)$ shows that we can take $v=1$. The above equations are therefore equivalent to $A=u^2f(t)$ and $B=u^3g(t)$ with $u=b$ and $t=a/b$. (Note: a solution with $b=0$ is impossible, as that would give $4A^3+27B^2=0$.)
\end{proof}

\begin{proposition}
Let $G=\bZ/2\bZ \times \bZ/2\bZ$. Then $X^{1/3} \lesssim N'_G(X) \lesssim X^{1/3}$.
\end{proposition}

\begin{proof}
This follows from the above two propositions. Note that, in the notation of Proposition~\ref{bd-var}, we have $m=n=1$, and so $(m+1)/6n=1/3$.
\end{proof}

\section{Asymptotics for the trivial group, $\bZ/2\bZ$, and $\bZ/3\bZ$}\label{sec:asymptotics}

In this section, we derive asymptotics for $N_G(X)$ when $G$ has order at most 3. The result is stated in Theorem~\ref{lem:main_asymp_lemma} below and is a refinement of Theorem~\ref{thm:asymptotics}. We take a unified approach by first counting integer points in explicit semi-algebraic sets using the Principle of Lipschitz \cite{davenport} and then sieving. Note that these asymptotics imply Theorem~\ref{thm:main} for these cases.

\begin{lemma}
Let $E/\bQ$ be an elliptic curve given by $y^2=x^3+Ax+B$ with $A$ and $B$ in $\bZ$. Then,
\begin{enumerate}
	\item\label{lempart:2tors} $E$ has a rational point of order 2 if and only if there exists $b \in \bZ$ such that $B=b^3+Ab$;
	\item\label{lempart:3tors} $E$ has a rational point of order 3 if and only if there exists $a, b \in \bZ$ such that $A=6ab+27 a^4$ and $B=b^2-27a^6$.
\end{enumerate}
\end{lemma}

\begin{proof}
As is well-known, $E$ has a rational 2-torsion point if and only if $x^3+Ax+B$ has a rational root. This yields part (\ref{lempart:2tors}). Part (\ref{lempart:3tors}) is proved in \cite[\S2]{GST}
\end{proof}

This lemma suggests we study integer points in the following three semi-algebraic regions:
\begin{align*}
	R_1(X)&=\{(a,b)\in\bR^2:|a|<X^{1/3}\text{ and }|b|<X^{1/2}\},\\
	R_2(X)&=\{(a,b)\in\bR^2:|a|<X^{1/3}\text{ and }|b^3+ab|<X^{1/2}\}, \textrm{and}\\
	R_3(X)&=\{(a,b)\in\bR^2:|6ab+27a^4|<X^{1/3}\text{ and }|b^2-27a^6|<X^{1/2}\}.
\end{align*}

The Principle of Lipschitz states that the number of integer points $r_{i}(X)$ in $R_i(X)$ is given by its area up to an error given by the maximal length of its projections onto the coordinate axes. From this we immediately see that $r_1(X)=4X^{5/6}+O(X^{1/2})$. The next lemma will allow us to compute the main term and the error for $i=2,3$. Before stating it, we must introduce several quantities.

Let $\alpha_\pm$ be the (unique) real root of $x^3\pm x-1$. Let $f_\pm=3x^4\pm6x^2+12x-1$, each of which has one negative and one positive root. We denote these four roots by
\begin{align*}
	&\alpha_4\approx0.08011\quad\text{and}\quad\alpha_1\approx-1.22259\quad\text{(roots of }f_+);\\
	&\alpha_3\approx0.08711\quad\text{and}\quad\alpha_0\approx-2.01637\quad\text{(roots of }f_-).
\end{align*}
Also, define $\alpha_2=-\sqrt{3}$ and $\alpha_5=\sqrt{3}$. Let $\displaystyle \beta_i=\sgn(\alpha_i)\sqrt{|\alpha_i|/3}$ except take $\beta_3$ to be negative. Then, $\beta_i<\beta_j$ for $i<j$. Define the hyperelliptic integrals
\[
	I_+=\int_{\beta_3}^{\beta_4}\sqrt{1+27a^6}da\approx0.33383
\]
and
\[
	I_-=\int_{\beta_0}^{\beta_1}\sqrt{-1+27a^6}da\approx0.32030.
\]
Finally, let
\[
	A_\pm(a)=\frac{\pm X^{1/3}-27a^4}{6a}\quad\text{ and }\quad B_\pm(a)=\sqrt{\pm X^{1/2}+27a^6}.
\]
We leave the verification of the following lemma to the reader.
\begin{lemma}\mbox{}
	\begin{enumerate}
		\item We have $(a,b)\in R_2(X)$ if and only if
\begin{align*}
& |b|<\alpha_+X^{1/6}\text{ and }|a|^3<X\text{, or} \\
& \alpha_+X^{1/6}\leq|b|<\alpha_-X^{1/6}\text{ and }-X^{1/3}<a<\frac{X^{1/2}}{|b|}-b^2.
\end{align*}
Thus,
\[
	\Area(R_2(1))=2\log(\alpha_-/\alpha_+)+\tfrac{4}{3} (\alpha_++\alpha_-).
\]
		\item Let $R^+_3(X)=R_3(X)\cap\{b\geq0\}$. Then, $(a,b)\in R^+_3(X)$ if and only if
\begin{displaymath}
\beta_iX^{1/12}<a<\beta_{i+1}X^{1/12} \qquad \textrm{and} \qquad g_i(a)<b<f_i(a),
\end{displaymath}
for some $0\leq i\leq 4$, where $f_i(a)$ and $g_i(a)$ are as in the following table.
\begin{center}
{\tabulinesep=1.2mm
\begin{tabu}{|c||c|c|c|c|c|} \hline
$i$ & 0 & 1 & 2 & 3 & 4 \\
\hline
$f_i$ & $A_-(a)$ & $A_-(a)$ & $A_-(a)$ & $B_+(a)$ & $A_+(a)$ \\
\hline
$g_i$ & $B_-(a)$ & $A_+(a)$ & 0 & 0 & 0 \\
\hline
\end{tabu}
}
\end{center}
Thus,
\[
	\Area(R_3^+(1))=I_+-I_-+\frac{1}{6}\log\left(\displaystyle\frac{\beta_0\beta_1\beta_5}{\beta_2\beta_3\beta_4}\right)+\frac{9}{8}\left(\beta_0^4+\beta_2^4+\beta_4^4-\beta_1^4-\beta_3^4-\beta_5^4\right).
\]
Furthermore, $(a,b)\in R_3(X)$ if and only if $(-a,-b)\in R_3(X)$. In particular, $|b|<2\sqrt{7}X^{1/4}$ for all $(a,b)\in R_3(X)$.
	\end{enumerate}
\end{lemma}
%\begin{table}[!h]
%\caption{Bounds for $R^+_3(X)$}
%{\tabulinesep=1.2mm
%\begin{tabu}{|c|c|c|} \hline
%		$i$	&	$f_i(a)$	&	$g_i(a)$	\\ \hline\hline
%		0	&	$A_-(a)$	&	$B_-(a)$	\\ \hline
%		1	&	$A_-(a)$	&	$A_+(a)$	\\ \hline
%		2	&	$A_-(a)$	&	0	\\ \hline
%		3	&	$B_+(a)$	&	0	\\ \hline
%		4	&	$A_+(a)$	&	0	\\ \hline
%\end{tabu}}
%\label{t:figi}
%\end{table}

Let $(d_1,d_2,d_3)=(6/5,2,3)$ and $(e_1,e_2,e_3)=(2,3,4)$. Note that for each $i$, the region $R_i(X)$ is homogeneous in $X$ in the sense that
\[
	\Area(R_i(X))=X^{1/d_i}\Area(R_i(1)).
\]
Therefore, combining the above lemma with the Principal of Lipschitz yields
\begin{equation}\label{eqn:Ri_count}
	r_i(X)=\Area(R_i(1))X^{1/d_i}+O(X^{1/e_i}).
\end{equation}
All that remains is to address the overcounting that occurs in $R_i(X)$. We do this in three lemmas.
\begin{lemma}
	Define $(A,B)=T_i(a,b)$ by
	\begin{enumerate}
		\item $T_1(a,b)=(a,b)$, if $i=1$;
		\item $T_2(a,b)=(a,b^3+ab)$, if $i=2$;
		\item $T_3(a,b)=(6ab+27a^4,b^2-27a^6)$, if $i=3$.
	\end{enumerate}
	Then, the number of $(a,b)\in R_i(X)\cap\bZ^2$ such that $4A^3+27B^2=0$ is $O(X^{1/6})$.
\end{lemma}
\begin{proof}
	Suppose $4A^3+27B^2=0$. First, consider $i=1$. Then, $4a^3=-27b^2$, which implies $a^\prime=-a/3\in\bZ$ and $b^\prime=b/2\in\bZ$. Since ${a^\prime}^3={b^\prime}^2<X$, the number of $(a,b)$ with discriminant zero is on the order of the number of sixth powers less than $X$, namely it is $O(X^{1/6})$. For $i=2$, B\'{e}zout's theorem says there are at most 3 pairs $(a,b)$ giving $(A,B)$; for $i=3$, the number is 24. Appealing to the argument for $i=1$ again yields the bound $O(X^{1/6})$.
\end{proof}
In view of this lemma, we may (and do) redefine $R_i(X)$ by removing from it the discriminant zero locus, while preserving the truth of \eqref{eqn:Ri_count}.
\begin{lemma}
	The number of $(A,B)$ for which there is more than one $(a,b)\in R_i(X)\cap\bZ^2$ with $(A,B)=T_i(a,b)$ is $O(X^{1/e_i})$.
\end{lemma}
\begin{proof}
	This is non-trivial only for $i=2,3$. Consider first $i=2$. Given an $E/\bQ$ with equation $y^2=x^3+Ax+B$ and an integer $b$ such that $B=b^3+Ab$, one obtains a non-trivial 2-torsion point $(-b,0)\in E(\bQ)$. Therefore, if there is another $b^\prime\in\bZ$ such that $B=b^{\prime3}+Ab^\prime$, then the torsion subgroup contains $\bZ/2\bZ\times\bZ/2\bZ$. But we proved in \S\ref{sec:Zmod2timesZmod2} that the number of such elliptic curves is $O(X^{1/3})$. The case $i=3$ is similar. Indeed, each pair $(a,b)\in\bZ^2$ such that $(A,B)=T_3(a,b)$ yields a pair of non-trivial 3-torsion points $(3a^2, \pm(9a^3+b))\in E(\bQ)$, which are negatives of each other. Ignoring the cases where $ab=0$ (which is $O(X^{1/4})$ many), a second pair $(a^\prime,b^\prime)$ that gives the same $(A,B)$ yields two \emph{new} non-trivial rational 3-torsion points. But this is impossible: an elliptic curve over $\bQ$ cannot contain two copies of $\bZ/3\bZ$ in its Mordell--Weil group.
\end{proof}
Let $\fE_i(X)=\{\text{equations }E_{A,B}:y^2=x^3+Ax+B:(A,B)=T_i(a,b),(a,b)\in R_i(X)\cap\bZ^2\}$. The above lemma implies that $\#\fE_i(X)=r_i(X)+O(X^{1/e_i})$. We now sieve out non-minimal equations to make our way from $\fE_i(X)$ to $N_{\bZ/i\bZ}(X)$.
\begin{theorem}\label{lem:main_asymp_lemma}
	For $i=1,2,$ and $3$,
	\[
		N_{\bZ/i\bZ}=\frac{\Area(R_i(1))}{\zeta(12/d_i)}X^{1/d_i}+O(X^{1/e_i}).
	\]
\end{theorem}
\begin{proof}
	Given an equation $E_{A,B}\in\fE_i(X)$, let $d$ be the largest twelfth power dividing $\gcd(A^3,B^2)$. Then, the curve defined by $E_{A,B}$ is also given (uniquely) by the minimal equation $E_{d^{-4}A,d^{-3}B}$. Sieving yields
	\begin{align*}
		N_{\bZ/i\bZ}^\prime(X)	&=\sum_{d=1}^{X^{1/12}}\mu(d)\#\fE_i(d^{-12}X)\\
							&=\sum_{d=1}^{X^{1/12}}\mu(d)\left(\Area(R_i(1))\frac{X^{1/d_i}}{d^{12/d_i}}+O(d^{-12/d_i}X^{1/e_i})\right)\\
							&=\Area(R_i(1))X^{1/d_i}\sum_{d=1}^{X^{1/12}}\frac{\mu(d)}{d^{12/d_i}}+O(X^{1/e_i})\\
							&=\frac{\Area(R_i(1))}{\zeta(12/d_i)}X^{1/d_i}+O(X^{1/e_i}).
	\end{align*}
	Since $N_{\bZ/i\bZ}(X)-N_{\bZ/i\bZ}^\prime(X)$ is within the error, the theorem is proved.
\end{proof}
	As remarked in the introduction, the case $i=1$ already appears in \cite[Lemma~4.3]{brumer}. In \cite[\S2]{grant}, the case $i=2$ is shown with the error bound $O(X^{\frac{1}{3}+\epsilon})$, for all $\epsilon>0$. The case $i=3$ appears to be new, and we have found no further cases in the literature either.

\end{document}